\documentclass[a4paper]{article}

\usepackage{a4wide}

\usepackage{graphicx}
\usepackage[T1]{fontenc}

\usepackage{amssymb}
\usepackage{amsmath}
\usepackage{amsthm}
\usepackage{ifthen}				
\usepackage{bm}
\usepackage[linesnumbered,ruled,vlined]{algorithm2e}
\usepackage{tikz}
\usetikzlibrary{decorations.pathreplacing}
\usepackage{enumerate}
\usepackage{pgfplots}
\pgfplotsset{compat=1.14}

\newtheoremstyle{it_dotless} 
                        {0.5em}   
                        {0.5em}   
                        {\itshape}  
                        {}          
                        {\bfseries} 
                        {:}         
                        {\newline}  
                        {}			

\newtheoremstyle{dotless} 
                        {0.5em}   
                        {0.5em}   
                        {}  		
                        {}          
                        {\bfseries} 
                        {:}         
                        {\newline}  
                        {}			
                        
\theoremstyle{it_dotless}
\newtheorem{theorem}{Theorem}
\newtheorem{lemma}[theorem]{Lemma}
\newtheorem{corollary}[theorem]{Corollary}

\theoremstyle{dotless}
\newtheorem{remark}[theorem]{Remark}
\newtheorem{example}[theorem]{Example}

\newtheorem{definition}[theorem]{Definition}

\usepackage[colorinlistoftodos,prependcaption,textsize=tiny]{todonotes}
\usepackage{xargs}

\newcommandx{\rewrite}[2][1=]{\todo[linecolor=red,backgroundcolor=red!25,bordercolor=red,#1]{#2}}
\newcommandx{\improve}[2][1=]{\todo[linecolor=yellow,backgroundcolor=yellow!25,bordercolor=yellow,#1]{#2}}
\newcommandx{\recheck}[2][1=]{\todo[linecolor=green,backgroundcolor=green!25,bordercolor=green,#1]{#2}}

\newcommand{\bitem}{\begin{itemize}}
\newcommand{\eitem}{\end{itemize}}
\newcommand{\mc}[1]{\mathcal{#1}}

\newcommand{\N}{\mathbb{N}}
\newcommand{\R}{\mathbb{R}}

\newcommand{\bpm}{\begin{pmatrix}}
\newcommand{\epm}{\end{pmatrix}}
\newcommand{\bsm}{\left(\begin{smallmatrix}}
\newcommand{\esm}{\end{smallmatrix}\right)}
\newcommand{\T}{\top}

\newcommand{\la}{\langle}
\newcommand{\ra}{\rangle}

\DeclareMathOperator{\tr}{tr}
\DeclareMathOperator{\Diag}{Diag}
\DeclareMathOperator{\diag}{diag}

\DeclareMathOperator{\supp}{supp}

\DeclareMathOperator{\conv}{conv}

\DeclareMathOperator{\rank}{rank}

\newcommand{\ifnonempty}[1]{\ifthenelse{\equal{#1}{\empty}}{}{(#1)}}

\newcommand{\vect}[1]{\bm{#1}}
\newcommand{\matr}[1]{\mathbf{#1}}
\newcommand{\I}{\;\middle | \;}
\newcommand{\e}[1][]{\vect{e}_{#1}}

\newcommand{\graphs}[1]{\mc{G}_{#1}}					                
\newcommand{\Pot}[2][]{2^{#2}_{#1}}           	                    
\newcommand{\PiM}[1]{\operatorname{SPM}^{#1}}                      
\newcommand{\ProM}[3]{\operatorname{CPM}_{#1}^{#2}\ifnonempty{#3}} 
\newcommand{\RroM}[3]{\operatorname{RPM}_{#1}^{#2}\ifnonempty{#3}} 
\newcommand{\AssM}[3]{\mc{U}_{#2,#1}\ifnonempty{#3}}                

\newcommand{\spct}{\operatorname{spct}}      
\DeclareMathOperator*{\trace}{tr}
\newcommand{\Pt}[3][]{\mc{C}_{#1}^{#2}\ifnonempty{#3}}              

\newcommand{\coll}{\operatorname{col}}      

\usepackage{hyperref} 	
\hypersetup{
    colorlinks=true,
    linkcolor=blue,
    filecolor=magenta,      
    urlcolor=cyan,
    citecolor=red
}

\begin{document}

\title{
A variant of the Lov\'{a}sz-Theta number based on projection matrices
\thanks{
The authors gratefully acknowledges support by the German Research Foundation (DFG), grant GRK 1653.
}
}

\author{Francesco Silvestri
 \footnote{Institut f\"{u}r Informatik, Heidelberg University, INF205, 69120 Heidelberg, Germany}$^{\; ,}$\footnote{IWR, Heidelberg University, INF205, 69120 Heidelberg, Germany}}

\maketitle

\begin{abstract}
We introduce a new model for the chromatic number $\chi(G)$ based on what we call combinatorial projection matrices, which is a special class of doubly stochastic symmetric projection matrices. Relaxing this models yields an SDP whose optimal value is the projection theta number $\hat{\vartheta}(G)$, which is closely related to the Szegedy number $\vartheta^+(G)$, a variant of the Lov\'{a}sz theta number.
We characterize that in general, $\hat{\vartheta}(G)\leq \vartheta^+(G)$, with equality if $G$ is vertex-transitive. While this seems to imply that working with binary matrices is a better paradigm than working with binary eigenvalues in this context, our approach is slightly faster than computing the Szegedy number on vertex-transitive graphs.

\end{abstract}

\section{Introduction}

\subsubsection*{Background}

The chromatic number $\chi(G)$ of a graph $G$ is the minimum number of colours needed to colour the nodes of $G$ in such a way that adjacent nodes receive distinct colours. A stable set in $G$ is a set of nodes $S\subseteq V$ such that no pair of nodes is adjacent, so a colouring is a partition of $V$ into stable sets. Both characterizing stable sets and computing $\chi(G)$ are classical NP-hard problems \cite{schrijver03}. Furthermore, it is even NP-hard to approximate $\chi(G)$ within $|V (G)|^{1/14-\epsilon}$ for any $\varepsilon> 0$ \cite{Bellare94}.

In the standard Integer Program Formulation of this problem \cite{Diaz01}, each stable set $S_j$ is represented by a characteristic vector $x_j$ and the number of such vectors being used is minimized. This gives the problem a highly symmetrical structure, as any permutation of the colours $\{x_j\}_{j\in[k]}$ results in new colourings of equal objective. As a consequence, methods based on convex relaxations like the LP relaxation usually perform badly unless they break these symmetries. Unfortunately, even though symmetry breaking constraints are available \cite{Diaz01}\cite{FaenzaK09}, LP relaxations usually still suffer from poor bounds due to the hardness of the problem.

On the other hand, the colouring problem can also be formulated as a rank constrained matrix problem which implicitly models the colouring $X$ as a sum of matrices $X_j$ corresponding to stable sets \cite{50years}. This removes the symmetry, as any permutation of the colours $X_j$ will only change the order of summation, which does not change the colouring $X$. This formulation can then be relaxed to an SDP relaxation and the resulting bound $\overline{\vartheta}(G)$ is well-know as the Lov\'{a}sz theta number $\vartheta(G)$ applied on the complement graph \cite{Lovasz79}. In general, and in particular for perfect graphs, $\overline{\vartheta}(G)$ is usually much better than the LP bound \cite{schrijver03}.

\subsubsection*{Summary}

In this paper introduce and analyse a new formulation for $\chi(G)$ based on combinatorial projection matrices, which is constructed to admit an SDP relaxation. 
The resulting SDP relaxation is closely related to the Szegedy-number, a variant of the Lov\'{a}sz theta number, which we compare against. 
While the latter relaxes matrices with binary entries, our approach relaxes binary eigenvalues instead, and we are able to prove explicit relations between both relaxations.\\

\subsubsection*{Related Work}
A lot of work has been dedicated to variants of $\overline{\vartheta}(G)$ with the hope of improving towards $\chi(G)$. Early work mostly consists of adding linear inequalities based on non-negativity and triangle inequalities, summarized in \cite{Dukanovic07}. 
Other work \cite{Rendl08}\cite{Lasserre15} shows that converging SDP hierarchies of growing sizes are available, but these relaxations quickly grow too big for actual computations.
Finally, \cite{Laurent08} compares the preceding results in a unified way while also using a reduction of colouring onto the stable set problem to yield even better bounds.\\

\subsubsection*{Organization of the paper}
In Section \ref{sec:colouring}, we shortly recall the relevant results about colouring. 
Section \ref{sec:combinatorial_projection_matrices} reviews what we call combinatorial projection matrices and states a few basic properties. 
In Section \ref{sec:projection_theta} we establish a new formulation for the colouring number $\chi(G)$ based on combinatorial projection matrices and introduce the projection $\vartheta$-number $\hat{\vartheta}(G)$ as SDP relaxation of this formulation. 
Additionally, we compare $\hat{\vartheta}(G)$ with the Szegedy-number $\vartheta^+(G)$ analytically and show that in general, $\vartheta^+(G)\leq \hat{\vartheta}(G)$,  while both numbers agree on vertex-transitive graphs.
Finally, Section \ref{sec:MM_Extention} investigates if the method of moments can help to close the gap between the two relaxations. 

\subsection{Notation}

Small letters like $a,x,\lambda$ are scalars or elements of a set and capital letters like $S,T,U$ describe sets. Small bold letters are used for vectors like $\vect{a}, \vect{x}$, $\vect{\lambda}$ and capital bold letters for matrices like $\matr{A}, \matr{X}, \matr{\Lambda}$. Instead of writing index sets of small size like $\{i\}$ or $\{i,j\}$, we may write $i$ or $ij$ respectively. 

For any $n\in \N$, let $[n]:=\{1,\ldots, n\}$. $\graphs{n}$ denotes the set of simple graphs $G$ whose node set is $[n]$ and whose edge set is $E(G)$. The adjacency matrix of $G\in \graphs{n}$ is denoted by $\matr{A}_G$. 

Given a set $I\subseteq [n]$, its \emph{characteristic vector}\index{characteristic vector} $\e[I]\in \{0,1\}^n$ is coordinate-wise defined as
\begin{equation*}
(\e[I])_i = \begin{cases} 1 &\text{if } i\in I, \\ 0 & \text{else.}\end{cases}
\end{equation*}
Geometrically, the map $I\mapsto \e[I]$ bijectively maps the power set $\Pot{n}$ to the vertices of the $n$-dimensional unit cube $[0,1]^n$. If $n$ is clear from context, we will write $\e[]$ instead of $\e[{[n]}]$. The matrix $\matr{I}_n$ is the $n$-dimensional identity, $\matr{J}_{n_1,n_2}=\e[{[n_1]}]\e[{[n_2]}]^\T$ is the matrix of all-ones, and we might write $\matr{J}_n$ or simply $\matr{J}$ instead of $\matr{J}_{n,n}$ if clear from context. 
For a matrix $\matr{A}$, $\coll(\matr{A})$ denotes the set of its columns, and if it is quadratic, $\spct(\matr{A})$ denotes its spectrum and $\tr(\matr{A})$ its trace. 

$\R^n_+$ is the cone of vectors in $\R^n$ with non-negative entries with conic order $\leq$ and $\mc{S}^n_+$ is the cone of positive semidefinite $n\times n$ matrices with conic order $\preceq$. The standard simplex is denoted as
\begin{equation*}
\Delta^n := \left\{ \vect{x}\in \R^n_+ \I \la \vect{x}, \e[]\ra =1\right\}.
\end{equation*}
The Kronecker-delta $\delta_{\Psi}$ assumes the truth values 0/1 of the statement $\Psi$, where $\delta_{i,j}$ is short for $\delta_{i=j}$.



\section{Preliminaries}\label{sec:colouring}

This section serves as a short recapitulation of the basics about graph colouring. For an extensive survey on convex relaxations for graph colouring, consider \cite{Laurent08}.

\subsection{Stable Sets}

\begin{definition}
A \emph{stable set}\index{stable set} of a graph $G\in \graphs{n}$ is a subset $S\subseteq [n]$ such that the subgraph induced by $S$ does not contain any edges. The set of stable sets in $G$ is denoted by $S_G\subseteq \Pot{[n]}$ and satisfies the properties of an independence system. The \emph{stable set number}\index{stable set!number} $\alpha(G)$ is the biggest size $|S|$ of any stable set in $G$ and defines a function $\alpha: \graphs{n}\rightarrow \N$. 
The \emph{clique number} $\omega(G)$ of $G$ is the biggest size $|S|$ of a stable set in the complement graph $G$, so $\omega(G):=\alpha(\overline{G})$. 
\end{definition}

Formally, we can express $\alpha(G)$ as the solution to an integer problem by working with characteristic vectors $\vect{x}\in \{0,1\}^n$ via
\begin{align}
\alpha(G)
&=\max  \left\{ \la \vect{x},\e\ra \I x_i x_j= 0\quad \forall (i,j)\in E,\, \vect{x}\in\{0,1\}^n\right\}\nonumber \\
&=\max \left\{ \la \vect{x},\e\ra \I \vect{x}^\T \matr{A}_G \vect{x} = 0 ,\, \vect{x}\in \{0,1\}^n \right\},\label{eq:STABLE}
\end{align}
where $\matr{A}_G$ denotes the adjacency matrix of $G$. 

%

\subsection{Graph Colourings}

\begin{definition}
A \emph{$k$-colouring}\index{colouring} of a graph $G\in \graphs{n}$ is a $k$-partition $\mc{T}=(T_1 , \ldots ,T_k )$ of the node set $[n]$, such that each $T_i$ is a stable set in $G$. The \emph{chromatic number}\index{chromatic number} $\chi(G)$ is the smallest value $k$ such that $G$ has a $k$-colouring. The set of $k$-colourings of $G$ is denoted by $\Pt[k]{}{G}$
\end{definition}


In the following, we will describe how to represent $k$-colourings as \emph{assignment matrices}\index{assignment matrix}  by treating their characteristic vectors as columns of an $n\times k$ binary matrices. To this end, let $S_G^*:=S_G\setminus \{\emptyset\}$ to denote
\begin{equation*}
V(S_G)= \left\{ \e[I]\in \{0,1\}^n \I I\in S_G^*\right\}=\left\{ \vect{x}\in \{0,1\}^n \I \supp(\vect{x})\in S_G^* \right\}
\end{equation*}
be the vertex-set associated with $S_G$ and define the set of $S_G$-constrained assignment matrices 
\begin{equation}
\AssM{k}{n}{S_G} = \left\{ \matr{U} \in \{0,1\}^{n\times k} \I \matr{U}\e[{[k]}]=\e[{[n]}],\; \coll(\matr{U})\subseteq  V(S_G) \right\}. \label{eq:AssM_Definition}
\end{equation}
Additionally, we will also use the shorthand
\begin{equation*}
\AssM{k}{n}{}:=\AssM{k}{n}{\Pot{[n]}} = \left\{ \matr{U} \in \{0,1\}^{n\times k} \I \matr{U}\e[{[k]}]=\e[{[n]}],\; \vect{0}\notin \coll(\matr{U}) \right\}.
\end{equation*}
With this notation at hand, we can express the colouring number simply as
\begin{equation*}
\chi(G)= \min\left\{ k \I \AssM{k}{n}{S_G}\neq \emptyset\right\}.
\end{equation*}

\begin{example}\label{ex:assignment_and_partition}
Let $G\in \graphs{3}$ be the graph given by the adjacency matrix 
\begin{equation*}
\matr{A}_G:=\bpm 0 & 0 & 1\\ 0 & 0 & 1\\ 1 & 1 & 0\epm.
\end{equation*}
Then $\left\{\{1,2\}, \{3\}\right\}$ is a valid colouring, and ordering the corresponding characteristic vectors lexicographically leads to the assignment matrix
\begin{equation*}
\matr{U}= \bpm 
1 & 0 \\ 
1 & 0 \\ 
0 & 1
\epm \in \AssM{2}{3}{S_G}.
\end{equation*}
\end{example}

The following hardness result shows that an explicit description of neither $\AssM{k}{n}{S_G}$ nor $S_G$ can easily be found:

\begin{theorem}[\cite{schrijver03,Bellare94}]\label{thm:hardness_colouring}
Computing $\alpha(G)$ or $\chi(G)$ is \emph{NP}-hard. Furthermore, it is \emph{NP}-hard to approximate $\chi(G)$ within $n^{1/14-\epsilon}$ for any $\varepsilon> 0$.
\end{theorem}

In spite of being NP-hard to compute in general, $\chi(G)$ behaves much nicer restricted to special classes of graphs with regards to the following construction.

\begin{definition}[Lov\'{a}sz theta number \cite{Lovasz79}]
The \emph{Lov\'{a}sz theta number}\index{Lov\'{a}sz theta number} $\vartheta(G)$ is the optimal value of the following primal/dual SDP problems
\begin{align*}
\max &\left\{ \la \matr{X},\matr{J}_n\ra \I  \la \matr{X}, \matr{I}_n\ra =1,\, \la \matr{X}, \matr{A}_{G}\ra =0,\, \matr{X}\succeq \matr{0}\right\},
\tag{$\vartheta$-P}\label{eq:PRIMAL}\\ 
\min &\left\{k \I \diag(\matr{Y})=\e, \, \la \matr{Y}, \matr{A}_{\overline{G}}\ra =0,\, \bpm k & \e^\T \\ \e & \matr{Y}\epm\succeq \matr{0} \right\}.
\tag{$\vartheta$-D}\label{eq:DUAL}
\end{align*}
\end{definition}

\begin{theorem}[Lov\'{a}sz sandwich Theorem \cite{Lovasz87}]\label{thm:Lovasz_sandwich}
For all $G\in \graphs{n}$, we have
\begin{equation*}
\omega(G)\leq \vartheta(\overline{G})\leq \chi(G).
\end{equation*}
\end{theorem}

This theorem is a celebrated result for \emph{perfect graphs}\index{graph!perfect}, where $\omega(H)=\chi(H)$ for all induced subgraphs $H\subseteq G$, and $\chi(G)$ can be computed in polynomial time. For more on the theory of perfect graphs, consider the recent survey \cite{Trotignon15}.

Despite the hardness results of Theorem \ref{thm:hardness_colouring}, a lot of publications are dedicated to improve $\vartheta(G)$ for imperfect graphs. We exemplarily introduce the following simple improvements taken from \cite{Dukanovic07} and \cite{Laurent08}, where several other relaxations for $\chi(G)$ may be found.

\begin{remark}
Adding $\matr{X}\geq \matr{0}$ to \eqref{eq:PRIMAL} and $\matr{Y}\geq \matr{0}$ to \eqref{eq:DUAL} results in the Schrijver number $\vartheta^-(G)$ \cite{Schrijver79} and the Szegedy number $\vartheta^+(G)$ \cite{Szegedy94} respectively. In particular, we get
\begin{equation*}
\omega(G)\leq \vartheta^-(\overline{G}) \leq \vartheta(\overline{G})\leq \vartheta^+(\overline{G})\leq \chi(G).
\end{equation*}
\end{remark}

\section{Combinatorial projection matrices}\label{sec:combinatorial_projection_matrices}

For the rest of this paper, let $\matr{R}=\left( \vect{r}_1,\ldots, \vect{r}_n\right) = \{r_{ij}\}_{i,j\in [n]} \in \R^{n\times n}$. We first recall some basic properties about projection matrices arising from orthogonal projections onto subspaces. 

\begin{definition}
Let $\matr{X}\in \R^{n\times k}$ be a matrix with full column rank to define its corresponding \emph{projection matrix}\index{projection matrix} as
\begin{equation*}
\rho(\matr{X}):=\matr{X} (\matr{X}^\T \matr{X})^{-1} \matr{X}^{\T}\in \R^{n\times n}.
\end{equation*}
The set of \emph{symmetric projection matrices} of size $n$ is defined as 
\begin{equation*}
\PiM{n}:=\left\{ \matr{R}\in \R^{n\times n} \I \matr{R}^2=\matr{R}, \, \matr{R}=\matr{R}^\T  \right\}\subseteq \mc{S}^n_+.
\end{equation*}
\end{definition}

Note that by assumption, $\matr{X}^\T \matr{X}$ is invertible and $\rho(\matr{X})\in~\PiM{n}$. We also recall some important properties of projection matrices:

\begin{lemma}\label{lem:projection_matrices_properties}
Let $\matr{R}\in \PiM{n}$ and $\matr{X}\in \R^{n\times k}$ with full column rank. Then the following holds:
\begin{enumerate}[(i)]
\item Binary eigenvalues: $\spct(\matr{R})\subseteq \{0,1\}$,
\item Rank equation: $\tr(\matr{R})=\rank(\matr{R})$,
\item Projection property: $\rho(\matr{X})\matr{X} =\matr{X}$, $\rho(\matr{X})\cdot \R^n=\matr{X}\cdot \R^k$,
\item Isometry invariance: $\rho(\matr{X})=\rho(\matr{XQ})$ for all orthogonal $\matr{Q}\in \R^{n\times n}$.
\end{enumerate}
\end{lemma}
\begin{proof}
$(i)$ follows from the Cayley-Hamilton theorem applied to the matrix polynomial $\matr{R}^2=\matr{R}$. For $(ii)$, note that $\tr(\matr{R})$ is equivalent to the sum of its eigenvalues, which count the rank due to $(i)$. $(iii)$ and $(iv)$ follow directly from the definition.
\end{proof}

\begin{example}
In Example \ref{ex:assignment_and_partition}, we considered the assignment matrix of the 2-colouring $\left\{ \{1,2\},\{3\}  \right\}$ given by
\begin{equation*}
\matr{U}=
\bpm 
0 & 1 \\ 
0 & 1 \\ 
1 & 0
\epm.
\end{equation*}
Applying $\rho$ to $\matr{U}$ leads us to the corresponding projection matrix
\begin{equation*}
\rho(\matr{U})= 
\bpm 
1 & 0 \\ 
1 & 0 \\ 
0 & 1 \\ 
\epm
\bpm 
2 & 0 \\ 
0 & 1 \\ 
\epm^{-1}
\bpm 
1 & 1 & 0 \\ 
0 & 0 & 1 \\ 
\epm
=
\bpm 
\tfrac{1}{2} & \tfrac{1}{2} & 0 \\ 
\tfrac{1}{2} & \tfrac{1}{2} & 0 \\ 
0 & 0 & 1
\epm.
\end{equation*}

\noindent Denoting the columns as $\rho(\matr{U})=\bpm \vect{r}_1, \vect{r}_2, \vect{r}_3\epm$, we observe for the following that
\bitem 
\item $\matr{U}^\T \matr{U}$ is a diagonal matrix which contains the size of parts of the partition,
\item the columns $\{\vect{r}_1,\vect{r}_2,\vect{r}_3\}$ consist of all eigenvectors of $\rho(\matr{U})$ with repetition,
\item $\{\vect{r}_1,\vect{r}_2,\vect{r}_3\}\subseteq \Delta^3$,
\item  $\|\vect{r}_i\|_0 \cdot \|\vect{r}_i\|_\infty =1$ holds for each column $i\in [3]$,
\item $\vect{r}_i = \vect{r}_j$ if and only if $(\vect{r}_i)_j=(\vect{r}_j)_i > 0$.
\eitem
\end{example}

The preceding example motivates the following definition of a special class of doubly stochastic matrices.

\begin{definition}\label{def:PiM}
The set of \emph{combinatorial projection matrices}\index{combinatorial projection matrices} $\ProM{k}{n}{S_G}$ is given as
\begin{equation*}
\ProM{k}{n}{S_G} = \left\{ \matr{R}\in \R^{n\times n} \I
\begin{array}{rlr}
\coll(\matr{R})                       &\subseteq\Delta^n_{S_G}                        \\
r_{ij}\cdot (\vect{r}_i -\vect{r}_j) &= \vect{0}            & \forall i,j\in [n] \\
\matr{R}                             &=\matr{R}^\T                               \\
\trace(\matr{R})                     &=k
\end{array}
\right\},
\end{equation*}
where
\begin{equation*}
\Delta^n_{S_G} := \left\{ \vect{r}\in \Delta^n \I \supp(\vect{r})\in S_G \right\}.
\end{equation*}
The quadratic equations 
\begin{equation*}
 r_{ij} \cdot (\vect{r}_i -\vect{r}_j) = \vect{0}\quad \forall i,j\in [n]
\end{equation*}
will be called \emph{block-inducing}\index{block-inducing constraints} in the following.
\end{definition}

As expected, this set contains the projection matrices corresponding to assignment matrices, which is shown by the following two lemmas.

\begin{lemma}\label{lem:double_stochastic_projection}
Combinatorial projection matrices are projection matrices
\begin{equation*}
\ProM{k}{n}{S_G}\subseteq \PiM{n}
\end{equation*}
and have strictly positive diagonal $\diag(\matr{R})>\vect{0}$. In particular, each column $\vect{r}_i$ is uniquely determined by its support via
\begin{equation*} 
r_{ji} = \delta_{j\in \supp(\vect{r}_i )} \frac{1}{\| {\vect{r}_i }\|_0}
\end{equation*}
and satisfies the equation
\begin{equation*}
\|\vect{r}_i\|_0 \cdot \|\vect{r}_i\|_\infty = 1.
\end{equation*}
\end{lemma}
\begin{proof}
Let $\matr{R}\in \ProM{k}{n}{S_G}$ and choose any $i,j,l\in [n]$. By assumption, the block-inducing equations $r_{il}\cdot (\vect{r}_i - \vect{r}_l) = \vect{0}$  and $r_{ij}\cdot (\vect{r}_i - \vect{r}_j)=\vect{0}$ show the identity
\begin{equation*}
r_{il} r_{jl} = r_{il} r_{ij} =  r_{ij} r_{il} = r_{ij} r_{jl},
\end{equation*}
using the symmetry of $\matr{R}$.
Then $\matr{R}^2=\matr{R}$ follows from
\begin{equation*}
(\matr{R}^2)_{ij}= \sum_{l\in [n]} r_{il} r_{jl} = \sum_{l\in [n]} r_{ij} r_{jl} = r_{ij} \Big(\sum_{l\in [n]}  r_{jl}\Big) = r_{ij}=(\matr{R})_{ij}\quad \forall i,j\in [n].
\end{equation*}
Furthermore, the block-inducing equations $r_{ij}\cdot (\vect{r}_i - \vect{r}_j)=\vect{0}$ show $r_{ii}(r_{ij}-r_{ii})=0$ for all $j\in [n]$, which implies $r_{ij}\in \{0, r_{ii}\}$. Since $\vect{r}_i\in \Delta^n$, we necessarily have $r_{ii}\cdot \|\vect{r}_i\|_0=~1$, which completes the proof.
\end{proof}

\begin{theorem}\label{thm:assignment_projection_bijection}
The combinatorial projection matrices are precisely the projection matrices corresponding to assignment matrices:
\begin{equation*}
\rho(\AssM{k}{n}{S_G}) = \ProM{k}{n}{S_G} \subseteq \PiM{n}.
\end{equation*}
\end{theorem}
\begin{proof}
We first show $\rho(\AssM{k}{n}{S_G})\subseteq \ProM{k}{n}{S_G}$, so let $\matr{U}=(\vect{u}_1,\ldots, \vect{u}_k)\in \AssM{k}{n}{S_G}$. We already know $\rho(\matr{U})\in \PiM{n}$ and the corresponding properties from Lemma \ref{lem:projection_matrices_properties}, so we only need to check the first two properties outlined in Definition \ref{def:PiM}. 

For the first, assume $(\vect{u}_j)_i=1$ to see
\begin{equation*}
\rho(\matr{U})\e[i] = \matr{U}(\matr{U}^\T \matr{U})^{-1} \matr{U}^\T \e[i] = \matr{U}(\matr{U}^\T \matr{U})^{-1} \e[j] = \matr{U} \frac{\e[j]}{\la \vect{u}_j, \vect{u}_j \ra}  = \frac{\vect{u}_j}{\la \e, \vect{u}_j \ra}\in \Delta_{S_G}^n,
\end{equation*}
where we used that $\vect{u}_j$ is binary in the second to last equation. 

For the second, assume $\e[i]^\T \rho(\matr{U}) \e[j]>0$, since $\rho(\matr{U})\geq 0$ and there is nothing to show otherwise. Assuming $(\vect{u}_{j'})_{j}=1$ and $(\vect{u}_{i'})_{i}=1$, we have
\begin{equation*}
0 < \e[i]^\T \rho(\matr{U}) \e[j]
= \e[i]^\T \matr{U}(\matr{U}^\T \matr{U})^{-1} \matr{U}^\T \e[j]
= e_{i'}^\T (\matr{U}^\T \matr{U})^{-1} e_{j'} = \delta_{i', j'}\cdot (\matr{U}^\T \matr{U})^{-1}_{i',i'},
\end{equation*}
since $(\matr{U}^\T \matr{U})^{-1}$ is diagonal. It follows that $i'=j'$, which in turn means 
\begin{equation*}
\matr{U}^\T \e[i] = \e[i']=\e[j']= \matr{U}^\T \e[j].
\end{equation*}
In particular,
\begin{equation*}
\rho(\matr{U})(\e[i] - \e[j]) = \matr{U}(\matr{U}^\T \matr{U})^{-1} \left(\matr{U}^\T (\e[i]-\e[j])\right) = \vect{0},
\end{equation*}
which we wanted to show.

The reverse inclusion $\rho(\AssM{k}{n}{S_G})\supseteq \ProM{k}{n}{S_G}$ follows if we can show that $\rho$ is surjective.
To show this, we argue that for $\matr{R}\in \ProM{k}{n}{S_G}$, the map
\begin{equation*}
\psi: \matr{R}\rightarrow \psi(\matr{R})=:\left\{ \supp(\vect{r}_i) \I i\in [n]\right\}\in \Pt[k]{}{G}
\end{equation*}
is well-defined and $\rho \circ \psi$ yields the identity on $\ProM{n}{k}{S_G}$.

Since $\vect{r}_i\in \Delta^n_{S_G}$, we have $\psi(\matr{R})\subseteq S_G^*$. Disjointness of the sets in $\psi(\matr{R})$ follows from the implications of the block-inducing equations, and coverage of $[n]$ follows due to $\matr{R}$ being doubly stochastic and so $\psi(\matr{R})\in \Pt[k]{}{G}$ is well-defined.

Finally, the map $\rho \circ \psi$ is the identity on $\ProM{n}{k}{S_G}$ since $\matr{R}$ is uniquely reconstructed from $\psi(\matr{R})$ by starting with $i\in \supp(\vect{r}_i)$ and applying Lemma \ref{lem:double_stochastic_projection}.
\end{proof}

\subsection{Convexification}\label{sec:convex_hull_projection_matrices}

In the next section, we will reduce the computation of $\chi(G)$ to the task of optimizing a linear function over $\ProM{k}{n}{S_G}$, which shows that a compact description of $\conv(\ProM{k}{n}{S_G})$ is most likely out of reach due to Theorem \ref{thm:hardness_colouring}. For this reason, we are interested to in studying relaxations of this convex hull. 

The set $\ProM{k}{n}{S_G}$ is constructed to make an SDP relaxation immediately available by replacing the non-linear constraints
\begin{equation*}
r_{ij} \cdot (\vect{r}_i - \vect{r}_j) = \vect{0}\quad \forall i,j\in [n]
\end{equation*}
with the psd. constraint $\matr{R}\succeq \matr{0}$, which is implied by the constraint $\matr{R}^2=\matr{R}$ from being a projection matrix.
We thus get the set
\begin{equation*}
\RroM{k}{n}{S_G} := \left\{ \matr{R}\in \R_+^{n\times n} \I \tr(\matr{R})=k,\; \matr{R}\e=\e,\; \la \matr{R}, \matr{A}_G\ra =0,\,    \matr{R}\succeq \matr{0} \right\}.
\end{equation*}

\section{The projection \texorpdfstring{$\vartheta$}{theta}-number}\label{sec:projection_theta}

We now propose a new formulation in terms of projection matrices and relate it to $\vartheta(G)$.

\begin{theorem}
We have the following characterization of the chromatic number:
\begin{align*}
\chi(G)
&=\min\left\{k \I \ProM{k}{n}{S_G}\neq \emptyset \right\}\\
&=\min\left\{\tr(\matr{R}) \I  
\begin{array}{rlr}
 \coll(\matr{R})                         & \subseteq \Delta^n         \\
 r_{ij}                                 &=0              & \forall (i,j)\in E(G) \\         
 r_{ij}\cdot  (\vect{r}_i - \vect{r}_j) &=\vect{0}       & \forall i,j\in [n]    \\
 \matr{R}                               &=\matr{R}^\T                                    
\end{array}
\right\}.
\end{align*}
\end{theorem}
\begin{proof}
The first equation follows by definition and Theorem \ref{thm:assignment_projection_bijection}, so we will show the second. Let $\chi'(G)$ denote the optimal value of the second optimization problem and let $\mc{T}$ be a minimal colouring with $\chi(G)$ colours. Then the projection matrix $\matr{R}(\mc{T})$ corresponding to $\mc{T}$ is feasible for the second optimization problem, since $\supp(\matr{R}(\mc{T})\e[j]) \in S_G$ and $\diag(\matr{R}(\mc{T}))>\vect{0}$ imply $\e[i]^\T \matr{R}(\mc{T}) \e[j]=0$ whenever $(i,j)\in E(G)$, and so $\chi'(G)\leq \chi(G)$.

To see the other inequality, first note that $\chi'(G)\geq 0$ is integral, since whenever $\matr{R}$ is feasible, $\matr{R}\in \PiM{n}$ as shown in the proof of Lemma \ref{lem:double_stochastic_projection}, and so the eigenvalues of $\matr{R}$ are binary. 
It suffices to show that $\matr{R}\in \ProM{\tr(\matr{R})}{n}{S_G}$ then, and the only thing that is left to prove is $\supp(\vect{r}_i)\in S_G$. 

Suppose that $\supp(\vect{r}_i)\notin S_G$, so there is $(j,l)\in E(G)$ such that $(\vect{r}_i)_j, (\vect{r}_i)_l>0$. It follows from $r_{ij}\cdot (\vect{r}_i - \vect{r}_j)= \vect{0}$ that $(\vect{r}_j)_l = (\vect{r}_i)_l >0$, which contradicts $(\vect{r}_j)_l=0$.
\end{proof}

The preceding theorem is important in that it makes the constraint $\supp(\vect{r}_i)\in S_G$ tractable, so that we can approximate $\chi(G)$ by relaxing the set $\ProM{k}{n}{S_G}$. Following Section~\ref{sec:convex_hull_projection_matrices}, we can immediately state the following relaxation.

\begin{definition}
The projection $\vartheta$-number is given as
\begin{align}
\hat{\vartheta}(G):=& \min\left\{ k\in \R^n_+ \I \RroM{k}{n}{S_G} \neq \emptyset \right\}\nonumber \\
=&\min\left\{ \tr(\matr{R}) \I \matr{R}\succeq \matr{0},\, \matr{R}\e=\e,\, \la \matr{R},\matr{A}_{\overline{G}}\ra =0,\, \matr{R}\geq \matr{0}\right\}. \label{eq:1st_order_colouring}
\end{align}
\end{definition}

We should note that by non-negativity of $\matr{R}$, we have 
\begin{equation*}
\la \matr{R}, \matr{A}_G\ra =0\quad \Leftrightarrow\quad  r_{ij} =0 \quad \forall (i,j)\in E(G),
\end{equation*}
which we will prefer for notation. For the following, we explicitly recall the Szegedy number we have already seen in Section \ref{sec:colouring}, which is given as
\begin{equation}
\vartheta^+(G):=\min\left\{ x_0 \I \bpm x_0 & \e^\T \\ \e & \matr{X}\epm \succeq \matr{0}, \, \diag(\matr{X})=\e, \, \la \matr{X},\matr{A}_{\overline{G}}\ra =0,\,   \matr{X}\geq \matr{0}\right\}.\label{eq:LovaszStrengthening}
\end{equation}

It makes sense to compare the quality of both relaxations, since they have similar computational complexity and similar constraints. The following theorem shows that unfortunately, the classical approach is always at least as good as using projection matrices.

\begin{theorem}\label{thm:inferior_projection}
Let $G\in\graphs{n}$, then $\vartheta^+(G)\geq \hat{\vartheta}(G)$.
\end{theorem}
\begin{proof}
Let $x_0=\vartheta^+(G)$ and $\matr{X}$ be an optimal solution to \eqref{eq:LovaszStrengthening}. Since $\matr{X}$ is symmetrical and non-negative, it follows from \cite{Brualdi66} that there exists a diagonal matrix $\matr{D}=\Diag(\vect{d})$ with $\vect{d}>\vect{0}$ such that
\begin{equation*}
\bpm 1 & \vect{0} \\ \vect{0} & \matr{D} \epm
\bpm \vartheta^+(G) & \e^\T \\ \e & \matr{X}\epm
\bpm 1 & \vect{0} \\ \vect{0} & \matr{D} \epm
= \bpm \vartheta^+(G) & \vect{d}^\T \\ \vect{d} & \matr{DXD}\epm\succeq \matr{0},
\end{equation*}
where $\matr{R}:=\matr{DXD}$ is doubly stochastic, so $\matr{R}\e=\e$. In particular, $\matr{R}$ is feasible for \eqref{eq:1st_order_colouring}, and we have $\tr(\matr{R})=\|\vect{d}\|_2^2$. Using positive semidefiniteness, for any $\mu \in \R$, we have
\begin{equation*}
0 \leq \bpm -1 \\ \mu \vect{d} \epm^\T \bpm \vartheta^+(G) & \vect{d}^\T \\ \vect{d} & \matr{R}\epm \bpm -1 \\ \mu \vect{d} \epm = 
\vartheta^+(G) - 2\mu \|\vect{d}\|_2^2 + \mu^2 (\vect{d}^\T \matr{R} \vect{d})
\end{equation*}
which becomes
\begin{align*}
\vartheta^+(G)
&\geq 2\mu \|\vect{d}\|_2^2 - \mu^2 (\vect{d}^\T \matr{R} \vect{d}) 
    = \|\vect{d}\|_2^2 \left(2\mu - \mu^2 (\tfrac{\vect{d}^\T}{\|\vect{d}\|_2} \matr{R} \tfrac{\vect{d}}{\|\vect{d}\|_2})\right)\\
&\geq \|\vect{d}\|_2^2 \left(2\mu - \mu^2 \right) 
    = \tr(\matr{R}) \left(2\mu - \mu^2 \right).
\end{align*}
Since $\max\left\{2\mu - \mu^2 \I \mu\in \R \right\} =1$, it follows that $\matr{R}$ is a feasible solution to \eqref{eq:1st_order_colouring} that satisfies $ \vartheta^+(G)\geq \tr(\matr{R})$ and the theorem follows.
\end{proof}

To get a better understanding of the projection model, it would be beneficial to get a quantitative bound on the magnitude of the gap $\vartheta^+(G)-\hat{\vartheta}(G)$. Fortunately, we can state an explicit asymptotic lowerbound.

\begin{theorem}\label{thm:theta_asymptotic}
The worst case gap absolute $\vartheta^+(G)-\hat{\vartheta}(G)$ has asymptotic behaviour
\begin{equation*}
\max\left\{ \vartheta^+(G)-\hat{\vartheta}(G)\I G\in \graphs{n} \right\}=\Omega(n).
\end{equation*}
In particular,
\begin{equation*}
\limsup_{n\rightarrow \infty}\max\left\{\frac{ \vartheta^+(G)-\hat{\vartheta}(G)}{n}\I G\in \graphs{n} \right\} \geq \left(\frac{3}{\sqrt{2}}-2\right) \approx 0,1213.
\end{equation*}
\end{theorem}
\begin{proof}
We will explicitly construct a graph family for which the bound on the gap holds true asymptotically. To this end, consider for any two integers $n_1, n_2 \in \N$, the graph 
\begin{equation*}
G(n_1,n_2):=K_{n_1}\cup K_{n_2},
\end{equation*}
 which we define as the union of two complete graphs with $n_1$ and $n_2$ nodes respectively. In particular, the corresponding adjacency matrix is given as
\begin{equation*}
\matr{A}_{G(n_1,n_2)}=\bpm \matr{J}_{n_1}-\matr{I}_{n_1} & \matr{0} \\ \matr{0} & \matr{J}_{n_2}-\matr{I}_{n_2}\epm,
\end{equation*}
and since these graphs are perfect, we explicitly have
\begin{equation*}
\chi(G(n_1,n_2))=\vartheta^+(\overline{G(n_1,n_2)})=\omega(G(n_1,n_2))=\max(n_1,n_2).
\end{equation*} 
For $\hat{\vartheta}(\overline{G(n_1,n_2)})$, we will construct the optimal solution $\matr{R}$ in closed form. Since this graph has multiple symmetries, we can assume $\matr{R}$ to be symmetry invariant and para\-metrize the feasible set with only three parameters $\alpha,\beta, \gamma$ by setting
\begin{equation*}
\matr{R}= \bpm \alpha \matr{I}_{n_1} & \beta \matr{J}_{n_1,n_2}\\  \beta \matr{J}_{n_2,n_1} & \gamma \matr{I}_{n_2}\epm.
\end{equation*}
Now that the constraint $\la \matr{R},\matr{A}_{G(n_1,n_2)}\ra =0$ is satisfied, the remaining affine constraints turn into
\begin{equation*}
\begin{array}{lcr}
\matr{R}\geq \matr{0} & \Leftrightarrow & \alpha, \beta,\gamma\geq 0,\\
\matr{R}\e = \e  & \Leftrightarrow & \alpha + n_2 \beta =1, \quad \gamma + n_1 \beta =1.
\end{array}
\end{equation*}
Assuming $\alpha, \gamma \neq 0$ to make the condition $\matr{R}\succeq \matr{0}$ non-trivial, we can use the Schur complement to rewrite
\begin{align*}
\matr{R}\succeq \matr{0}  \quad &\Leftrightarrow \quad  \gamma \matr{I}_{n_2} - \frac{\beta^2}{\alpha} \matr{J}_{n_2,n_1} \cdot \matr{J}_{n_1,n_2}\succeq \matr{0} \quad 
 \Leftrightarrow\quad  \gamma \matr{I}_{n_2} - \frac{\beta^2}{\alpha} n_1 \matr{J}_{n_2}\succeq \matr{0} \\
& \Leftrightarrow\quad  \gamma \geq  \frac{\beta^2}{\alpha} n_1 n_2,
\end{align*}
where we used the fact that $\matr{J}_{n}$ only has one non-zero eigenvalue given by $n$. Lastly, we can explicitly express the objective function as
\begin{equation*}
\tr(\matr{R}) =\alpha n_1 + \gamma n_2 = n_1+n_2 - 2 n_1 n_2 \beta,
\end{equation*}
using the affine constraints. Ignoring the constants, the resulting problem of computing the number $\hat{\vartheta}(\overline{G(n_1,n_2)})$ is equivalent to
\begin{equation*}
\max\left\{ \beta \I \alpha +n_2\beta =1,\quad \gamma +n_1\beta =1,\quad
\alpha \gamma \geq \beta^2 n_1 n_2,\quad \alpha,\beta, \gamma \geq 0\right\}.
\end{equation*}
Using the equations, one can show that the unique solution is 
\begin{equation*}
\bpm \alpha \\ \beta \\ \gamma\epm =\frac{1}{n_1+n_2}\bpm n_1 \\  1 \\ n_2 \epm 
\end{equation*}
and
\begin{equation*}
\hat{\vartheta}(\overline{G(n_1,n_2)}) = \frac{n_1^2 +n_2^2}{n_1+n_2}.
\end{equation*}
In particular, we now have the gap 
\begin{equation}
\Delta(n_1,n_2):=\vartheta^+(\overline{G(n_1,n_2)})-\hat{\vartheta}(\overline{G(n_1,n_2)})= \max(n_1,n_2)-\frac{n_1^2 +n_2^2}{n_1+n_2}.\label{eq:theta_gap}
\end{equation}
W.l.o.g., let $n_1=m$, $n_2=\mu m\in \N$ for some $\mu\in [0,1]$. Then $\eqref{eq:theta_gap}$ reads
\begin{equation*}
\Delta(m,\mu m) = m - \frac{(1+\mu^2) m^2}{(1+\mu) m}= \mu \left(\frac{1-\mu}{1+\mu}\right) m. 
\end{equation*}
Finally, optimizing the choice $\mu\in[0,1]$ yields the biggest theoretical gap for $\mu= \sqrt{2}-1$ and
\begin{equation*}
\max_{\mu \in [0,1]} \Delta(m,\mu m)= (3-2\sqrt{2})m.
\end{equation*}
For growing $m$, we can approximate this gap arbitrarily well by choosing $n_1=m$ and $n_2= \lceil (\sqrt{2}-1) m \rceil$ to get a graph of size $\lceil \sqrt{2}m\rceil$ with asymptotic relative gap
\begin{equation*}
\limsup_{m\rightarrow \infty} \frac{\Delta(m,\lceil (\sqrt{2}-1)m\rceil)}{\lceil \sqrt{2}m\rceil}= \frac{3-2\sqrt{2}}{\sqrt{2}}\approx 0,1213.
\end{equation*}
\end{proof}

As a corollary of the preceding theorem, the following result sheds some light on the discrepancy between the two relaxations by relating $\hat{\vartheta}(G)$ to Theorem \ref{thm:Lovasz_sandwich}.

\begin{corollary}
The inequation $\omega(G)\leq \hat{\vartheta}(\overline{G})$ does not hold in general. In particular, $\hat{\vartheta}(G)$ is not necessarily exact for perfect graphs $G$. 
\end{corollary}

Despite this disadvantage over the classical formulation, we can show some useful properties for the class of vertex-transitive graphs. We first cite the property in question.

\begin{theorem}[{\cite{Szegedy94}}]\label{thm:vertex-transitive-theta}
For all $G\in \graphs{n}$, the inequality
\begin{equation*}
\vartheta^+(G)\cdot \vartheta^-(\overline{G})\geq n
\end{equation*}
holds, with equality if $G$ is vertex-transitive.
\end{theorem}

As shown in the following, this inequality can be sharped by using $\hat{\vartheta}(G)$.

\begin{lemma}\label{lem:coloring_product}
For all $G\in\graphs{n}$, the inequality $\hat{\vartheta}(G) \cdot \vartheta^{-}(\overline{G}) \geq n$ holds.
\end{lemma}
\begin{proof}
Let $\matr{R}$ be an optimal solution to \eqref{eq:1st_order_colouring}. Recalling the definition
\begin{equation*}
\vartheta^-(\overline{G})= \max \left\{ \la \matr{X},\matr{J}\ra \I \tr(\matr{X})=1,\,  \la \matr{X},\matr{A}_{\overline{G}}\ra =0, \, \matr{X}\succeq \matr{0}, \, \matr{X}\geq \matr{0} \right\}
\end{equation*}
as \eqref{eq:PRIMAL} with non-negativity constraints, we see that $\matr{X}:= \frac{1}{\tr(\matr{R})} \matr{R}$ is a feasible solution and as desired,
\begin{equation*}
\vartheta^-(\overline{G})\geq  \la \matr{X}, \matr{J}\ra = \frac{\la \matr{R}, \matr{J}\ra }{\tr(\matr{R})} = \frac{n}{\hat{\vartheta}(G)}.
\end{equation*}
\end{proof}

\begin{theorem}
For vertex-transitive $G\in \graphs{n}$, we have
\begin{equation*}
\vartheta^+(G) = \hat{\vartheta}(G).
\end{equation*}
In particular, $\omega(G)=\hat{\vartheta}(\overline{G})=\chi(G)$ for vertex-transitive perfect $G\in \graphs{n}$.
\end{theorem}
\begin{proof}
Using Theorem \ref{thm:inferior_projection}, Lemma \ref{lem:coloring_product} and Theorem \ref{thm:vertex-transitive-theta}, we see that
\begin{equation*}
\vartheta^+(G)\geq \hat{\vartheta}(G) \geq \frac{n}{\vartheta^-(\overline{G})} = \vartheta^+(G)
\end{equation*}
whenever $G\in\graphs{n}$ is vertex-transitive.
\end{proof}

Vertex-transitive graphs are known examples for which the relaxations $\vartheta(G)$ and $\vartheta^+(G)$ perform badly, and it is surprising to see that an analogue of Theorem \ref{thm:Lovasz_sandwich} for $\hat{\vartheta}(G)$ can be recovered in this case. In particular, any advantage of $\vartheta(G)$ over $\hat{\vartheta}(G)$ seems to be related to exploiting the lacking symmetries of a given graph.

\section{Extension with the Method of Moments}\label{sec:MM_Extention}

The relaxation $\vartheta(G)$ is inherently tied to the method of moments, a technique from polynomial optimization. This method constructs a hierarchy of convex relaxations of growing size for polynomial optimization problems, which converge towards the global optimum. Applying it to \eqref{eq:STABLE} yields a formulation that can be transformed into \eqref{eq:PRIMAL} as a first step in this hierarchy.

Following this approach, we can consider $\ProM{k}{n}{S_G}$ as an algebraic variety and apply the method of moments to it as well. For our purposes, it will be enough to consider the first stage of the relaxation hierarchy; we will skip the intricacies of the general method of moments and construct the relaxation of the first stage explicitly, which is motivated by the following line of thought:

Recalling that $\matr{R}=\left( \vect{r}_1,\ldots, \vect{r}_n\right)$, we would like to work with the rank-$1$ matrices
\begin{equation}
\matr{M}(\matr{R})=\bpm 1\\ \vect{r}_1 \\ \vdots \\ \vect{r}_n \epm \bpm 1\\ \vect{r}_1 \\ \vdots \\ \vect{r}_n \epm^\T,
\end{equation}
which contains all the products of entries of $\matr{R}$, since all constraints in $\ProM{k}{n}{S_G}$ turn into linear constraint over $\matr{M}$. Since this can not be done efficiently in practice, the method of moments proceeds by using all these linear constraints and applies it to the tractable matrix variable
\begin{equation*}
\matr{M}_1(\matr{R})=
\bpm 
1          & \vect{r}_1^\T & \ldots & \vect{r}_n^\T \\
\vect{r}_1 & \matr{R}_{11} & \ldots & \matr{R}_{1n} \\
\vdots     & \vdots        & \ddots & \vdots        \\
\vect{r}_n & \matr{R}_{n1} & \ldots & \matr{R}_{nn}
\epm \succeq \matr{0}.
\end{equation*}
instead of $\matr{M}(\matr{R})$, where the matrix blocks $\matr{R}_{ij}$ are new variables that linearize the products $\vect{r_i}\vect{r_j}^\T$. The linear constraints extracted from $\ProM{k}{n}{S_G}$ can be linewise converted to the system
\begin{align*}
\matr{R}_{ij} \e[] &= \vect{r}_{i}  & \la \vect{r}_{j},\e[]\ra &= 1\\
\matr{R}_{ii} \e[j] &= \matr{R}_{ij} \e[i] = \matr{R}_{jj} \e[i] = \matr{R}_{ji} \e[j] & \matr{R}_{ii}\e[i]&=\diag(\matr{R}_{ii})\\
\la \matr{R}_{ii}, \matr{A}_{\overline{G}}\ra &= 0 \\ 
\sum_{i\in [n]} \tr(\matr{R}_{ii})&=k
\end{align*}
for all choices of $i,j\in [n]$. By invoking the block inducing equations, one can show that additionally, we can recover the original matrix $\matr{R}$ from the $\matr{R}_{ii}$ through the equation
\begin{equation*}
\sum_{i\in [n]} \matr{R}_{ii} = \matr{R}.
\end{equation*}

Since the size of $\matr{M}(\matr{R})$ prohibitive, we relax it further by discarding all matrices $\matr{R}_{ij}$ with $i\neq j$. This relaxes $\matr{M}_1(\matr{R})\succeq \matr{0}$ to the systems
\begin{equation*}
\matr{M}_1(\matr{r}_i)=
\bpm 
1          & \vect{r}_i^\T \\
\vect{r}_i & \matr{R}_{ii}
\epm \succeq \matr{0} \quad \forall i\in [n],
\end{equation*}
and since $\bpm -1 & \e[n]^\T\epm^\T$ belongs to the kernel of these matrices, we can discard the first row and column from the formulation. This leads us to the strengthening
\begin{equation*}
\RroM{k}{n}{S_G}' = \left\{ \matr{R}\in F \I \tr(\matr{R})=k,\; \matr{R}\e=\e,\; \la \matr{R}, \matr{A}_{\overline{G}}\ra =0  \right\},
\end{equation*}
where 
\begin{equation*}
F := \left\{ \matr{R} \in \R^{n\times n} \I \matr{R} = \sum_{i\in [n]} \matr{R}_{ii}, \; (\matr{R}_{ii})_{jl}=(\matr{R}_{jj})_{il}, \; \matr{R}_{ii}\in F_i\; \forall i,j,l\in [n]\right\}
\end{equation*}
is build up from the individual sets
\begin{equation*}
F_i:=\left\{\matr{X}  \in \R_+^{n\times n} \I   \la \matr{X}, \matr{J}_n \ra =1,\;  \matr{X}\e[i]=\diag(\matr{X}),\; \matr{X}\succeq \matr{0} \right\}
\end{equation*}
that are distinguished by the constraint $\matr{X}\e[i]=\diag(\matr{X})$. In particular, the symmetry condition in $F$ makes sure that we can consider each matrix $\matr{R}_{ii}$ as a slice of a symmetrical third-order tensor, as shown in Figure \ref{fig:3rd_order_tensor}.

The advantage of $\RroM{k}{n}{S_G}'$ over $\RroM{k}{n}{S_G}$ is the inclusion $F\subseteq \R^{n\times n}_+ \cap \mc{S}_+^n$.

\begin{remark}
\cite{Laurent08} already proposed the construction of a symmetrical third-order tensor to relax a moment matrix of order $1$ for the combinatorial moment matrix of the stable set problem. From this perspective, each entry of our tensor can be identified with a set $\{i,j,l\}$ via its indices and understood as a rescaling of combinatorial moments up to the third order. 

While this process can be used to construct symmetrical tensors of higher orders from appropriate combinatorial moment matrices of rank $1$, this approach does not generalize in our setting. The reason for this is the assumption $\rank(\matr{R})=k>1$, which breaks the symmetry; the equation $(\matr{R}_{ij})_{lm} = (\matr{R}_{il})_{jm}$ will generally not hold, and so the entries associated with various permutations of $\{i,j,l,m\}$ can not be assumed to be equal.
\end{remark}

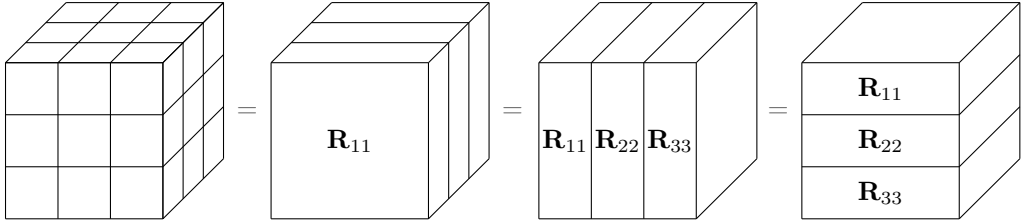
\begin{figure}[!ht]
\centering
\begin{minipage}{.2\textwidth}
\centering
\begin{tikzpicture}[scale=0.69,domain=0:3]
\foreach \x in{0,...,3}
{   \draw (0,\x ,3) -- (3,\x ,3);
    \draw (\x ,0,3) -- (\x, 3,3);
    \draw (3,\x ,3) -- (3,\x ,0);
    \draw (\x ,3,3) -- (\x ,3,0);
    \draw (3,0,\x ) -- (3,3,\x );
    \draw (0,3,\x ) -- (3,3,\x );
}
\end{tikzpicture}
\end{minipage}
=
\begin{minipage}{.2\textwidth}
\centering
\begin{tikzpicture}[scale=0.69,domain=0:3]
	\draw (1.5, 1.5, 3) node {$\matr{R}_{11}$};
	\draw (0, 0, 3) -- (3, 0, 3);
    \draw (3, 0, 3) -- (3, 0, 0);
    \draw (0, 0, 3) -- (0, 3, 3);
    \draw (0, 3, 3) -- (0, 3, 0);
    \draw (3, 3, 3) -- (3, 3, 0);    
\foreach \x in{0,...,3}
{   \draw (3,0,\x ) -- (3,3,\x );
    \draw (0,3,\x ) -- (3,3,\x );
}
\end{tikzpicture}
\end{minipage}
=
\begin{minipage}{.2\textwidth}
\centering
\begin{tikzpicture}[scale=0.69,domain=0:3]
	\draw (0.5, 1.5, 3) node {$\matr{R}_{11}$};
	\draw (1.5, 1.5, 3) node {$\matr{R}_{22}$};
	\draw (2.5, 1.5, 3) node {$\matr{R}_{33}$};		
    \draw (0, 0, 3) -- (3, 0, 3);
    \draw (3, 0, 3) -- (3, 0, 0);
    \draw (3, 0, 0) -- (3, 3, 0);
    \draw (0, 3, 0) -- (3, 3, 0);
    \draw (0, 3, 3) -- (3, 3, 3);    
\foreach \x in{0,...,3}
{   \draw (\x ,0,3) -- (\x ,3,3);
    \draw (\x ,3,3) -- (\x ,3,0);
}
\end{tikzpicture}
\end{minipage}
=
\begin{minipage}{.2\textwidth}
\centering
\begin{tikzpicture}[scale=0.69,domain=0:3]
	\draw (1.5, 2.5, 3) node {$\matr{R}_{11}$};
	\draw (1.5, 1.5, 3) node {$\matr{R}_{22}$};
	\draw (1.5, 0.5, 3) node {$\matr{R}_{33}$};		
    \draw (0, 0, 3) -- (0, 3, 3);
    \draw (0, 3, 3) -- (0, 3, 0);
    \draw (3, 0, 0) -- (3, 3, 0);
    \draw (0, 3, 0) -- (3, 3, 0);
    \draw (3, 0, 3) -- (3, 3, 3);
\foreach \x in{0,...,3}
{   \draw (0,\x ,3) -- (3,\x ,3);
    \draw (3,\x ,3) -- (3,\x ,0);
}
\end{tikzpicture}
\end{minipage}
\caption[Relaxation of projection matrix visualized as third order tensor]
{Different ways to align the variables of the various $\matr{R}_{ii}$ into the same symmetrical third order tensor, each cube representing one variable.}
\label{fig:3rd_order_tensor}
\end{figure}

We note that even by refining $\hat{\vartheta}(G)$ by using the computationally expensive $\RroM{k}{n}{S_G}'$ instead of $\RroM{k}{n}{S_G}$, we cannot guarantee a lower-bound of $\vartheta(G)$. To see this, let
\begin{equation*}
\hat{\vartheta}'(G):= \min\left\{ k\in \R^n_+ \I \RroM{k}{n}{S_G}' \neq \emptyset \right\}
\end{equation*}
be the corresponding strengthening of $\hat{\vartheta}(G)$. 

Empirical evidence shows that $\hat{\vartheta}'(G)=\chi(G)$ holds for the graph class $G(n_1,n_2)$ used as counter example in Theorem \ref{thm:theta_asymptotic}. However, this counter example can be generalized to find another class of perfect graphs for which $\vartheta(G)>\hat{\vartheta}'(G)$. Let $n_1,n_2,n_3\in \N$ such that 
\begin{equation*}
G(n_1,n_2,n_3)=K_{n_1}\cup K_{n_2}\cup K_{n_3}
\end{equation*}
is the union of three complete graphs. This class of graphs is perfect, and as such we get
\begin{equation*}
\vartheta(\overline{G(n_1,n_2,n_3)}) = \omega(G(n_1,n_2,n_3))= \max\{n_1,n_2,n_3\}.
\end{equation*}
Extending this construction, we can define another class of perfect graphs by taking the union of a complete graph on $n_1$ nodes together with $m$ isolated nodes. More formally,
\begin{equation*}
G(n_1,\e[m])=K_{n_1}\cup \left(\bigcup_{i\in [m]} K_{1}\right),
\end{equation*}
where 
\begin{equation*}
\vartheta(\overline{G(n_1,\e[m])}) = \omega(G(n_1,\e[m]))= n_1. 
\end{equation*}
Fixing the number of total nodes to $9$, the tables in Figure \ref{fig:theta_hat_prime} and \ref{fig:theta_hat_prime_trailing} show the behavior of the various relaxations. While the quality definitely increases while going from $\hat{\vartheta}(G)$ to $\hat{\vartheta}'(G)$, the difference between $\hat{\vartheta}'(G)$ and $\vartheta(G)$ grows roughly linear with the number of connected components.

We suspect that the bad performance of these relaxations is based on the fact that the relaxations are \emph{not monotone} in terms of subgraphs. In particular, if $H$ is a subgraph of $G$, then the implication
\begin{equation*}
H\leq G \quad \Rightarrow \quad \vartheta(H)\leq \vartheta(G),
\end{equation*}
is not true for either $\hat{\vartheta}(G)$ or $\hat{\vartheta}'(G)$. Since it is not easy to add this property to the functions in question, they are left at an inherent disadvantage compared to the original $\vartheta$-function. This begs the question if in general, approximating matrices with binary eigenvalues is harder than approximating matrices with binary entries.

\begin{table}[tb]
    \centering
    \begin{minipage}{.49\textwidth}
        \centering
		\begin{tabular}{|ccc|ccc|}
			\hline 
			 & & & & & \\[-10pt]
			$n_1$ & $n_2$ & $n_3$ & $\hat{\vartheta}$ & $\hat{\vartheta}'$ & $\vartheta=\chi$ \\
			\hline	
			3 & 3 & 3  &  3     & 3     & 3\\
			4 & 3 & 2  &  3.222 & 3.968 & 4\\
			4 & 4 & 1  &  3.666 & 4     & 4\\
			5 & 2 & 2  &  3.666 & 4.972 & 5\\
			5 & 3 & 1  &  3.888 & 4.983 & 5\\
			6 & 2 & 1  &  4.555 & 5.983 & 6\\
			7 & 1 & 1  &  5.666 & 6.985 & 7\\
			\hline
		\end{tabular}
		\caption{Relaxations for $G(n_1, n_2, n_3)$.}
		\label{fig:theta_hat_prime}
	\end{minipage}
    \begin{minipage}{.49\textwidth}
        \centering
		\begin{tabular}{|cc|ccc|}
			\hline
			 & & & & \\[-10pt]
			$n_1$ & $m$ & $\hat{\vartheta}$ & $\hat{\vartheta}'$ & $\vartheta=\chi$ \\
			\hline	
			2 & 7   &  1.222 & 1.772 & 2\\
			3 & 6   &  1.666 & 2.792 & 3\\
			4 & 5   &  2.333 & 3.851 & 4\\
			5 & 4   &  3.222 & 4.905 & 5\\
			6 & 3   &  4.333 & 5.951 & 6\\
			7 & 2   &  5.666 & 6.986 & 7\\
			8 & 1   &  7.222 & 8     & 8\\
			\hline
		\end{tabular}
		\caption{Relaxations for $G(n_1,\e[m])$.}
		\label{fig:theta_hat_prime_trailing}
	\end{minipage}
\end{table}

\section{Conclusion}

We introduced a new model for the chromatic number $\chi(G)$ based on what we call combinatorial projection matrices and proved its correctness. Relaxing this model yielded an SDP whose optimal value was defined as the projection theta number $\hat{\vartheta}(G)$, which was closely related to the Szegedy number $\vartheta^+(G)$, a variant of the Lov\'{a}sz theta number.
We characterized that in general, $\hat{\vartheta}(G)\leq \vartheta^+(G)$, with equality if $G$ is vertex-transitive. Additionally, we gave examples for both equality and strict inequality. 

Furthermore, we investigated whether the application of the method of moments could help in bridging this gap and showed with counter examples that in general, this is not the case. 

These results seem to imply that working with binary matrices is a better paradigm than working with binary eigenvalues in this context. This is to be expected, since matrices with binary eigenvalues can be seen as generalization of binary variables, but it is unexpected in the sense that SDP relaxations are usually adept at approximating eigenvalue-problems. 

For future work, it will be interesting to closer inspect the relation between matrices with binary entries and matrices with binary eigenvalues. It would be interesting to know whether one can use one modeling paradigm and convert it to the other, and in particular, if SDP relaxations based on binary matrix-entries are always superior to SDP relaxations based on binary eigenvalues.

\bibliographystyle{plain}
\bibliography{colouring,papers}

\end{document}